\documentclass[11pt]{amsart}
\usepackage{amsmath,mathtools}
\usepackage{amssymb}
\newtheorem{prop}{Proposition}[section]

\newtheorem{defn}[prop]{Definition}

\newtheorem{lem}[prop]{Lemma}
\newtheorem{cor}[prop]{Corollary}

\newtheorem{thm}[prop]{Theorem}

\newtheorem{rem}[prop]{Remark}

\newtheorem{fact}[prop]{Fact}
\newtheorem{obs}[prop]{Observation}
\numberwithin{equation}{section}

\def\BZ{\mathbb{Z}}

\def\Z{\mathbb{Z}}
\def\Q{\mathbb{Q}}

\def\N{\mathbb{N}}

\def\mod{\mathop{\rm mod}\nolimits}

\DeclarePairedDelimiter{\ceil}{\lceil}{\rceil}
\DeclarePairedDelimiter{\floor}{\lfloor}{\rfloor}

\usepackage[dvipsnames]{xcolor}
\usepackage{hyperref}

\begin{document}

\title{Strongly Obtuse Rational Lattice Triangles}

\author{Anne Larsen, Chaya Norton, and Bradley Zykoski}
\address{Department of Mathematics\\
    Harvard University \\
    Cambridge, MA 02138 \\
    USA}
\email{larsen@college.harvard.edu}

\address{Department of Mathematics \\
University of Michigan \\
Ann Arbor, MI 48109 \\
USA}
\email{nchaya@umich.edu}

\address{Department of Mathematics \\
University of Michigan \\
Ann Arbor, MI 48109 \\
USA}
\email{zykoskib@umich.edu}

%

\begin{abstract}
We classify rational triangles which unfold to Veech surfaces when the largest angle is at least $\frac{3\pi}{4}$. When the largest angle is greater than $\frac{2\pi}{3}$, we show that the unfolding is not Veech except possibly if it belongs to one of six infinite families. Our methods include a criterion of Mirzakhani and Wright that built on work of M\"oller and McMullen, and in most cases show that the orbit closure of the unfolding cannot have rank 1.
\end{abstract}

\maketitle

\section{Introduction}

The question considered in this paper is motivated by the following simple problem: what can be said about the dynamical system consisting of a billiard ball bouncing around a polygonal billiard table?

One approach to this problem uses the method of unfolding described in \cite{ZK} to transform the piecewise linear billiard path on a rational polygonal table (i.e., a table whose angles are rational multiples of $\pi$) into a straight path on a translation surface known as the ``unfolding" of the polygon, where the dynamics of straight-line flow might be better understood. For example, Veech \cite{Veech} proved that any translation surface whose affine automorphism group is a lattice has ``optimal dynamics" (i.e., straight-line flow in any given direction is either completely periodic or uniquely ergodic) and that the unfolding of an obtuse isoceles triangle with angles of the form $(\frac{\pi}{n}, \frac{\pi}{n}, \frac{(n-2)\pi}{n})$ is such a surface. Similar methods were later used to determine whether other rational triangles also have this ``lattice property," and several more individual lattice triangles and families of rational triangles have been identified.

Combining Veech's characterization of lattice triangles with the fact that the orthic triangle provides a known periodic billiard trajectory in each acute triangle, Kenyon and Smillie \cite{KS} were able to formulate a number theoretic criterion for the angles of an acute rational triangle that would be satisfied by any lattice triangle. They used this criterion to classify all acute and right-angled rational lattice triangles, up to the conjecture that their computer search had identified all triples not satisfying the criterion. This number theoretic conjecture was then proved by Puchta \cite{Puchta}, completing the classification.

Less is known for obtuse triangles, as there is no obvious periodic billiard trajectory in this case and so the methods of \cite{KS} do not easily extend. After the family of isoceles triangles originally described by Veech, the family of lattice triangles with angles $(\frac{\pi}{2n}, \frac{\pi}{n}, \frac{(2n-3)\pi}{n})$ was discovered independently in \cite{Vorobets} and \cite{Ward}, and more recently, computer-assisted computations of billiard trajectories led Hooper \cite{Hooper} to identify one more obtuse lattice triangle, the $(\frac{\pi}{12}, \frac{\pi}{3}, \frac{7\pi}{12})$, and to conjecture that the list of known obtuse lattice triangles is now complete. We prove half of this conjecture:

\begin{thm}
\label{realmain}
A rational obtuse triangle with obtuse angle $\ge 135^\circ$ has the lattice property if and only if it belongs to one of the two known families $(\frac{\pi}{n}, \frac{\pi}{n}, \frac{(n-2)\pi}{n})$ and $(\frac{\pi}{2n}, \frac{\pi}{n}, \frac{(2n-3)\pi}{2n})$.
\end{thm}

Our work on obtuse lattice triangles builds not so much on previous approaches to the problem as on recent results of a more complex analytical nature, using the equivalent characterization of a lattice triangle as one whose unfolding generates a Teichm\"uller curve. M\"oller \cite{Moller} proved that the period matrix of a Teichm\"uller curve has a block diagonal form, a result later extended by Filip \cite{Filip} to the period matrix of any rank 1 orbit closure. Combining this fact with an application of the Ahlfors variational formula \`a la McMullen \cite[Theorem 7.5]{McMullen}, Mirzakhani and Wright \cite[Theorem 7.5]{MW}
gave a condition under which the unfolding of a triangle must have orbit closure of rank $\ge 2$. From this, one can immediately derive a simple number-theoretic criterion that all obtuse rational lattice triangles must satisfy.

The main technical result of this paper is a classification of all obtuse rational triangles with obtuse angle $> \frac{2\pi}{3}$ for which it is possible to apply the \cite{MW} criterion (this cutoff being chosen because there are significantly more triangles with smaller obtuse angle for which it is not possible).
This entails a detailed case analysis, but a key ingredient is the use of approximation by rational numbers of small denominator, combined with known estimates on a particular number-theoretic function (the Jacobsthal function, which was used by Puchta \cite{Puchta} and McMullen \cite{McMullen2} in related classifications). By these means, we give a (computer-assisted) proof of the following theorem:

\begin{thm}
\label{main}
An obtuse rational triangle with obtuse angle $> 120^\circ$ satisfies the \cite{MW} criterion if and only if it does not belong to one of six (infinite) one-parameter families of triangles and is not one of seven exceptional triangles.
\end{thm}

Two of these families are known families of lattice triangles. The computer program of R\"uth, Delecroix, and Eskin \cite{RDE} has shown that the seven exceptional triangles do not have the lattice property. And by finding parallel cylinders of incommensurable moduli on the unfoldings, we are able to prove that one of the remaining families is not a family of lattice triangles. This allows us to complete the classification of rational obtuse lattice triangles with obtuse angle $\ge 135^\circ$ and prove Hooper's conjecture in this case.

The paper is organized as follows. In \S2, we derive the number-theoretic criterion in the needed form from \cite{MW}. In \S3, we establish some preliminary results about the solution sets of the inequalities derived in \S2 and explain the problems that arise when the obtuse angle is $\le 120^\circ$. In \S4, we outline the case analysis that will be used to prove Theorem \ref{main}. This case analysis will be carried out in \S\S5--8. A description of the computer search and its results is in \S9. Finally, \S10 contains the geometric argument used to rule out the remaining family of triangles with angle $\ge 135^\circ$, proving Theorem \ref{realmain}.

\textbf{Acknowledgements:} This work was done during the (online) 2020 University of Michigan REU. We are very grateful to Alex Wright for suggesting the problem and for his help and guidance. We would also like to thank Alex Eskin for kindly offering to use his program, currently under development by Vincent Delecroix and Julian R\"uth, to check our exceptional triangles, as well as the first triangle in each infinite family we found.

\section{Derivation of the Criterion}

In this section, we will briefly explain the setup in \cite[\S\S6--7]{MW} and explain how \cite[Theorem 7.5]{MW} implies the criterion we will state in \ref{derivation}.

First of all, we will set a standard form to refer to rational triangles. We write $(p,q,r)$, with $p,q,r \in \N, \,\, p \le q \le r, \,\,\gcd(p,q,r)=1$, to refer to the triangle with angles $(\frac{p\pi}{n}, \frac{q\pi}{n}, \frac{r\pi}{n})$, where $n=p+q+r$. (The gcd condition ensures that the choice of $p,q,r$ is unique, and the $p \le q \le r$ condition fixes the order.)

Now, the unfolding of the triangle $(p,q,r)$ is the translation surface $(X,\omega)$ where $X$ is the normalization of the curve $y^n = z^p(z-1)^q$ with holomorphic differential $\omega = y^{-n+1} z^{p-1} (z-1)^{q-1} \, dz$. $X$ has the obvious automorphism $y \mapsto e^{2\pi i/n}y$, and the space of holomorphic 1-forms on $X$ can be broken into eigenspaces, where the eigenspace of eigenvalue $e^{2\pi ai/n}$ has dimension $\{\frac{-ap}{n}\} + \{\frac{-aq}{n}\} + \{\frac{-ar}{n}\} - 1$ (the notation $\{x\}$ meaning the fractional part of $x$, $x - \floor{x}$). An explicit basis for each eigenspace is described in \cite[Lemma 6.1]{MW}. Plugging eigenforms of eigenvalues $e^{2\pi ai/n}$ and $e^{2\pi bi/n}$ into the integral given by the Ahlfors-Rauch variational formula, \cite[Proposition 7.3]{MW} shows that the resulting variation of the period matrix is nonzero if and only if $a + b \equiv 2$ mod $n$. This fact is then applied in \cite[Theorem 7.5]{MW} to see that if there is $a \in (\Z/n)^\times$ with $2a \not\equiv 2$ mod $n$ such that both the $e^{a2\pi i/n}$ and $e^{(2-a)2\pi i/n}$ eigenspaces are nonzero, then the unfolding has orbit closure of rank $> 1$, as this corresponds to a nonzero off-diagonal derivative of the period matrix in what would otherwise be a diagonal block (by work of Filip \cite{Filip}).

\begin{prop}
\label{derivation}
The unfolding of an obtuse triangle $(p,q,r)$ in the notation described above does not have the lattice property if there exists some $a \in (\Z/n)^\times$ with $2a \not\equiv 2$ mod $n$ such that two of the following ``mod $n$" inequalities are satisfied:
\begin{equation}
\label{twoone}
ap < 2p, \,\,\,\, aq < 2q, \,\,\,\, ar < 2r
\end{equation}
\end{prop}
Let $[x]_n$ be the representative of the mod $n$ equivalence class of $x$ which is in $[0,n)$. Then, for example, the ``mod $n$" inequality $ap < 2p$ is satisfied if $[ap]_n < [2p]_n$. In what follows, we will almost entirely want to consider numbers ``mod $n$"; where it seems unlikely to cause confusion, we will drop the bracket notation, which tends to clutter up equations. In a similar spirit, we will write $x \in (\Z/n)^\times$ to mean ``$x$ is coprime to $n$."

\begin{proof}
By the discussion above, a triangle does not have the lattice property if there is some $a \in (\Z/n)^\times$ with $2a \not\equiv 2$ mod $n$ such that
$$\left\{\frac{-ap}{n}\right\} + \left\{\frac{-aq}{n}\right\} + \left\{\frac{-ar}{n}\right\} > 1$$
and
$$\left\{\frac{-(2-a)p}{n}\right\} + \left\{\frac{-(2-a)q}{n}\right\} + \left\{\frac{-(2-a)r}{n}\right\} > 1$$
(these being the conditions for the eigenspaces to be nonzero). We rewrite as follows: first of all, as $c(p+q+r) = cn \equiv 0$ mod $n$ for any $c$, the left-hand sides of these two inequalities must be integral, and as $\{x\} < 1$ (so each sum is $<3$), these inequalities are equivalent to
\begin{equation}
\label{twotwo}
[-ap]_n + [-aq]_n + [-ar]_n = 2n
\end{equation}
and
\begin{equation}
\label{twothree}
[(-2+a)p]_n + [(-2+a)q]_n + [(-2+a)r]_n = 2n
\end{equation}
Since $a$ is a unit, we have $[-ax]_n = n - [ax]_n$ for $x = p,q,r$ (this being true unless $ax \equiv 0$). We then note that
$$[(-2+a)x]_n = \begin{cases} [-2x]_n + [ax]_n & [ax]_n < [2x]_n \\ [-2x]_n + [ax]_n - n & [ax]_n \ge [2x]_n \end{cases}.$$
By Equation \ref{twotwo}, we have
$$[ap]_n + [aq]_n + [ar]_n = n$$
Assuming that $(p,q,r)$ is an obtuse triangle, we must have $r > \frac{n}{2}$, $p,q < \frac{n}{2}$, and so
$$[-2p]_n = n - 2p, \,\,\, [-2q]_n = n - 2q, \,\,\, [-2r]_n = 2n - 2r$$
So
$$[-2p]_n + [ap]_n + [-2q]_n + [aq]_n + [-2r]_n + [ar]_n = 3n$$
which means that in order to satisfy Equation \ref{twothree}, we must have $[ax]_n < [2x]_n$ for exactly two of $p,q,r$. And as
$$[2p]_n + [2q]_n + [2r]_n = 2p + 2q + 2r - n = n$$
it is impossible to have $[ax]_n < [2x]_n$ for all of $p,q,r$ (as the sum of the $[ax]_n$ would be between 0 and $n$), so we can replace ``exactly two" with ``two."
\end{proof}

For the rest of the paper, we will investigate how to find such an $a$. As we will see, the main difficulty is not finding elements of $\Z/n$ satisfying two of the inequalities, but ensuring that one of these elements is a unit. (The $2a \not\equiv 2$ condition is generally not a major consideration, although it is often part of the problem in the families of triples where there is no such $a$.) Initially, it was hoped that such an $a$ could always be found for triples with sufficiently large $p$ or $n$, but this has turned out not to be true. (See Proposition \ref{counterexamples}
for more details.)

\section{Preliminaries}
We start by defining notation and terminology that will be used throughout the paper.

\begin{defn}
It will often be useful to refer to the set of solutions of one of the inequalities described in Equation \ref{twoone}, so we set
$$S_p := \{a \in [0, n): [ap]_n < [2p]_n\}$$
(defining $S_q$ and $S_r$ similarly).
\end{defn}
It is worth noting that, in this obtuse case,
$$[2p]_n = 2p, \,\,\,\, [2q]_n = 2q, \,\,\,\, [2r]_n = 2r - n = n -2p - 2q$$

\begin{defn}
We will call a unit $a \in (\Z/n)^\times$ with the property $2a \not\equiv 2$ mod $n$ a ``usable" unit. 
\end{defn}
Clearly, there can be at most two unusable units, 1 and $\frac{n}{2}+1$. Of these, 1 is always unusable, and $\frac{n}{2}+1$ is unusable iff 4 divides $n$. (First of all, for $\frac{n}{2}+1$ to be an integer, $n$ must be even. And if $n$ is divisible by 2 but not 4, then $\frac{n}{2}+1$ is even, therefore not a unit. But if 4 divides $n$, $\frac{n}{2}+1$ is its own inverse.)

\begin{rem}
With these definitions, we can restate our problem in the following way: when does one of the intersections $S_p \cap S_q, S_p \cap S_r, S_q \cap S_r$ contain a usable unit?
\end{rem}

We start by considering usable units in $S_p$ or $S_q$.

\begin{rem}
\label{spsq}
Suppose $x < \frac{n}{2}$, and set $a = \gcd(x,n), x = ab, n = ac$. As $\gcd(b,c) = 1$, $b$ is invertible in $\Z/c$, so we can let $d$ be the representative in $[0,c)$ of the equivalence class of $b^{-1}$ in $\Z/c$. (In the future, we will write ``let $d = b^{-1} \in (\Z/c)^\times$".)
Then
$$S_x = \{kd + lc: 0 \le k < 2b, \, 0 \le l < a\}$$
\end{rem}

\begin{prop}
\label{usableunits}
If $x < \frac{n}{2}$, $S_x$ contains a usable unit if and only if none of the following is true: $x = 1$; $x = 2$ and $n$ is even; or $x = 4$ and $n \equiv 4$ mod 8.
\end{prop}

\begin{proof}
We break into cases based on $\gcd(x,n)$:

If $\gcd(x,n) = 1$, then if $1 < x$ $(< \frac{n}{2}+1)$, $x^{-1} \in S_x$ is a usable unit. (And if $x = 1$, then $S_1 = \{0,1\}$ does not contain a usable unit.)

If $1 < \gcd(x,n) < x$, then $S_x$ contains all elements of $\Z/n$ equivalent to $kd$ mod $c$, for $0 \le k < 2b$. Any unit mod $c$ is coprime to all prime divisors of $c$, so to be a unit mod $n$, it suffices to be coprime to all remaining prime divisors of $a$. In particular, there are at least $\phi(a)$ units mod $n$ equivalent to a given unit mod $c$, by the Chinese remainder theorem. (As usual, we use $\phi$ to denote Euler's totient function.) We know that there are at least two units mod $c$ of the form $kd, 0 \le k < 2b$ (namely, $1$ and $d$), so $\phi(a) \ge 2$ implies that $S_x$ contains $\ge 4$ units, of which $\ge 2$ must be usable. On the other hand, $a \ne 1$ and $\phi(a) < 2$ would imply $a = 2$, in which case the unusable units are $\equiv 1$ mod $c$, and so the one or two units in $S_x$ $\equiv d$ mod $c$ must be usable.

Finally, if $\gcd(x,n) = x$, we apply the same argument as in the previous case, except that the two known units 1 and $d$ coincide. So $\phi(a) \ge 3$ implies that $S_x$ contains $\ge 3$ units, of which $\ge 1$ must be usable, but one needs to consider the cases $\phi(a) < 3$, i.e., $a = 1,2,3,4,6$. For $a = 3,6$, $S_x$ contains $\frac{n}{3}+1, \frac{2n}{3}+1$, of which $\ge 1$ must be a usable unit. So the only cases in which $S_x$ might not contain a usable unit are $x=1,2,4$ (dividing $n$). As mentioned already, $S_1 = \{0,1\}$ does not contain a usable unit. If $x = 2$ and $n$ is even, then $S_2 = \{0,1, \frac{n}{2}, \frac{n}{2}+1\}$ does not contain a usable unit. And if $x = 4$ and $x$ divides $n$, then $S_4 = S_2 \cup \{\frac{n}{4}, \frac{n}{4}+1, \frac{3n}{4}, \frac{3n}{4}+1\}$, of which the potential usable units are $\frac{n}{4}+1, \frac{3n}{4} + 1$. If 8 divides $n$, these are units (as they are coprime to $\frac{n}{4}$, which shares the same prime factors as $n$), and otherwise, if $n \equiv 4$ mod $8$, these are even, so not units.
\end{proof}

\begin{rem}
\label{remark}
Unfortunately, there are more cases when $S_r$ does not contain any usable units. In fact, this is guaranteed to happen if $r = \frac{c}{2c-1}n$; multiples of $r$ are multiples of $\frac{1}{2c-1}n = 2r - n$, and so $[ar]_n < [2r]_n$ implies $[ar]_n = 0$, i.e., $S_r$ consists entirely of multiples of $2c - 1$.
\end{rem}
This remark is essential in the following proposition:

\begin{prop}
\label{counterexamples}
There is no constant $c$ such that the criterion of \cite{MW} can be applied to all triples with $p > c$.
\end{prop}

\begin{proof}
We describe a method of constructing arbitrarily large examples in which the \cite{MW} condition is not satisfied: Let $p$ be any prime $>2$, and let $m$ be the product of all numbers $<2p$ excluding $p$. As $\gcd(m,p) = 1$, take $c = m^{-1} \in (\Z/p)^\times$. Then consider triples of the form
$$n = (xp - c)m, \,\,\,\, q = \frac{(p - 1)n}{2p - 1} - p, \,\,\,\, r = \frac{pn}{2p - 1}$$
for any positive integer $x$. (As $2p - 1$ divides $m$, $q$ and $r$ are indeed integers.) By Remark \ref{remark}, $S_r$ does not contain a usable unit; it therefore suffices to show that $S_p \cap S_q$ does not contain a usable unit. But using the fact that
$$S_p = \{kp^{-1}: 0 \le k < 2p\}$$
from \ref{spsq}, we see that the only units $S_p$ contains are $1, p^{-1}$, and as 1 is not usable, we need only check that $p^{-1} \notin S_q$. However, note that $n \equiv -1$ mod $p$, so $\frac{n+1}{p} = p^{-1} \in \Z/n$, which implies
$$p^{-1} r = \frac{n+1}{p} \frac{pn}{2p - 1} = \frac{(n+1)n}{2p - 1} \equiv \frac{n}{2p - 1} \mod n$$
$$ \implies p^{-1} q \equiv n - \frac{n}{2p - 1} - 1 = \frac{2p - 2}{2p - 1}n - 1 > \frac{2p -2}{2p - 1}n - 2p = 2q$$
so $p^{-1} \notin S_q$.
\end{proof}

Experimentally, it seems that the triples with $\frac{n}{2} < r \le \frac{2n}{3}$ and $p \ne 1,2, 4$ (as $p=1,2,4$ leads to additional problems) for which the \cite{MW} criterion is not satisfied follow roughly this pattern, in that $p$ is a small prime not dividing $n$, $n$ is highly composite, and $r = \frac{c}{2c-1}n$ for some $c$. However, the above proposition by no means gives the complete list of such triples; because of this issue, the rest of the paper focuses on triangles with obtuse angle $> 120^\circ$ (i.e., with $r > \frac{2n}{3}$), where it has been possible to identify precisely the cases in which the criterion cannot be applied.

\section{Proof Outline}
A more detailed statement of our main technical result is as follows.

\begin{thm}
\label{maintwo}
An obtuse rational triangle $(\frac{p\pi}{n}, \frac{q\pi}{n}, \frac{r\pi}{n})$ with obtuse angle $> \frac{2\pi}{3}$ satisfies the \cite{MW} criterion (implying that the orbit closure of its unfolding has rank $\ge 2$) if and only if none of the following is the case:
\begin{enumerate}
\item $p = q = 1$
\item $p = 1$, $q = 2$, $n$ is even
\item $p = 1, q = 4, r \equiv 7$ mod $8$
\item $p = 1, r = 3q+1$
\item $p = 2, r = 3q + 2$
\item $p = 4, r = 3q + 4$
\item $(p,q,r)$ is one of the following triples: $(1,4,11), (1,3,16), (2,3,17), \\(1,4,21), (1,8,19), (3,8,29), (2,11,29)$
\end{enumerate}
\end{thm}

In fact, definitive results are already known for some of the triangles on this list. As mentioned in the introduction, families 1 and 2 are known to be families of lattice triangles (see \cite{Veech}, \cite{Vorobets}, \cite{Ward}). The first element of family 3, the triangle (1,4,7), is the lattice triangle found by \cite{Hooper}, although it does not have obtuse angle $> \frac{2\pi}{3}$ and therefore does not, strictly speaking, belong on our list. And the half of family 4 with $q$ odd (and $>1$) is proven not to have the lattice property in \cite[Theorem B]{Ward}. Furthermore, the unfoldings of all the exceptional triangles in item 7 have been checked by the computer program of R\"uth, Delecroix, and Eskin \cite{RDE} and found to have dense orbit closures.

As Theorem \ref{maintwo} is proved by a rather complicated case analysis followed by computer-checking of certain triples, we will explain here how all the parts fit together. (We will basically split into cases along two major axes, the size of $\gcd(q,n)$ and the size of $q$.)

In \S5, we will derive a new form for $S_p$, which will be particularly useful when some $\gcd$ condition is satisfied. In particular, we almost entirely deal with the case $\gcd(q,n) > 2$ and partially deal with the case $\gcd(q,n) = 2$. The main tool will be the Chinese remainder theorem (cf. the proof of Proposition \ref{usableunits}), but some extra complications do arise when $\gcd(q,n)$ is a small power of 2. (This is perhaps the case where the $2a \not\equiv 2$ requirement becomes most problematic.)

In \S6, we will finish the proof for $q \le \frac{\sqrt{n}}{2}$. Applying the results of the previous section, as well as the obvious but useful fact that $c(p+q+r) \equiv 0$ mod $n$ for all $c$, we will see that in this case, $S_p \cap (S_q \cup S_r)$ contains a usable unit if $S_p$ does. There are only a few special cases in which $S_p$ does not contain a usable unit, and in these cases (still assuming $q \le \frac{\sqrt{n}}{2}$), we prove that either $S_q \cap S_r$ contains a unit or $(p,q,r)$ belongs to families (1)--(3) in Theorem \ref{maintwo}.

The case $q > \frac{\sqrt{n}}{2}$ will be addressed in \S\S7--8. In \S7, we will quickly deal with the case $\gcd(q,n) > 2$ and then, for the cases $\gcd(q,n) = 1$ or $2$, we will give the first part of a rational approximation argument to prove that if $n$ satisfies certain bounds, $S_q \cap S_r$ must contain a unit. (These bounds are related to the Jacobsthal function, which will be introduced in this section.) Although the main idea is not very complicated, there are many special cases to be checked, corresponding to the scenarios in which $\frac{q^{-1}r}{n}$ (or its counterpart in the $\gcd(q,n) = 2$ case) is well-approximated by a fraction of denominator $\le 3$. The proofs of the special cases, which make up \S8, tend to use the same few ideas, but the approach and the resulting bound are slightly different each time, so that it does not seem possible to condense the proofs in any useful way.

Using the bounds obtained in \S\S7--8 and known bounds on the Jacobsthal function (which we will introduce in Definition \ref{jacobsthal}), together with a computer experiment that greatly decreased the size of the search space, we were able to reduce the proof of Theorem \ref{maintwo} to an easy computer calculation. (It will suffice to check all triples with $n \le 10000$.) The details of this are contained in \S9.

\section{Observations about the case $\gcd(q,n) > 1$}
We start with a new characterization of $S_p$ and $S_q$ (really, of $S_x$ for any $x < \frac{n}{2}$). Thinking of multiples of $x$ (mod $n$) ``jumping along" the number line from 0 to $n$, the first two jumps starting when $S_x$ passes 0 are in the ``target zone" $[0, 2x)$, and then the jumps leave the target zone until they reach $n$ again. Translating this into a formula, we have the following:

\begin{obs}
\label{origs}
If $x < \frac{n}{2}$,
$$S_x = \bigg\{ \ceil[\bigg]{\frac{nm}{x}}, \ceil[\bigg]{\frac{nm}{x}}+1: 0 \le m < x\bigg\}$$
(where as usual, $\ceil{r}$ is the smallest integer $\ge r$).
\end{obs}

This description will be most useful in \S6, but we have introduced it now because of the following corollary:

\begin{cor}
\label{factor}
If $x$ divides $y$ and $y < \frac{n}{2}$, then $S_x \subset S_y$. In particular, $S_{\gcd(q,n)} \subset S_q$, and
if $\gcd(p,q) > 1$, then $S_p \cap S_q \supset S_{\gcd(p,q)}$ contains the usable unit $\gcd(p,q)^{-1}$ (which exists since $\gcd(p,q,n) = 1$).
\end{cor}

This corollary is useful for the following lemma.

\begin{lem}
\label{gcdqnnottwo}
Suppose $l > 2$ is a prime factor of $q$ and $n$, and $r > \frac{2n}{3}$. Then $S_l \cap S_r$ (and hence $S_q \cap S_n$) contains a usable unit. (The same holds when $q$ is replaced everywhere by $p$.)
\end{lem}

\begin{proof}
As $l$ divides $n$,
$$S_l = \left\{ \frac{kn}{l}, \frac{kn}{l} + 1: 0 \le k < l \right\}$$
where everything of the form $\frac{kn}{l}$ is definitely not a unit. Now, if $l^2$ divides $n$, every $\frac{kn}{l}+1$ is a unit mod $n$, and if not, then $\frac{kn}{l}+1$ is a unit except when $k = -(\frac{n}{l})^{-1}$ mod $l$. (This follows from the Chinese remainder theorem and the fact that $\frac{kn}{l}+1 \equiv 1$ modulo every prime divisor of $n$ except for possibly $l$.) Of course, $k = 0$ gives the unit 1, which is non-usable, but as we stipulated $l > 2$, none of these are $\frac{n}{2}+1$, so every element of $\{\frac{kn}{l} + 1\}$ except 1 and potentially one other element is a usable unit. Now, as $\gcd(q,r,n) = 1$, we have $\gcd(r,l) = 1$, so for each $j$, there is some $k$ so that $\frac{kn}{l} r \equiv \frac{jn}{l}$.

Now, as
$$\frac{kn}{l}+1 \in S_r \iff \left[\frac{jn}{l} + r\right]_n \in [0, 2r-n) \iff \frac{jn}{l} \in [n-r, r) \supset \left[\frac{n}{3}, \frac{2n}{3}\right]$$
any $j \in \BZ/l$ with $\frac{j}{l} \in [\frac{1}{3}, \frac{2}{3}]$ corresponds to some element $\frac{kn}{l} + 1 \in S_l \cap S_r$. As $1 \notin S_r$ and there is only potentially one non-unit of the form $\frac{kn}{l} + 1$, if there are two $j$ such that $\frac{jn}{l} \in [\frac{n}{3}, \frac{2n}{3}]$, at least one of them must correspond to a usable unit in $S_l \cap S_r$. But this condition is clearly satisfied for $l \ge 3$.
\end{proof}

The previous lemma can be applied in all cases where $\gcd(q,n)$ has a prime factor other than 2. In the following two lemmas we will basically deal with the case when $\gcd(q,n)$ is a power of two $\ge 4$, although we will revisit this case the next two sections.

\begin{lem}
\label{poweroftwo}
Suppose $\gcd(q,n) = 2^m$ for some $m \ge 3$, $q > 8$, and $r > \frac{2n}{3}$. Then $S_q \cap S_r$ contains a usable unit. Similarly, if $\gcd(q,n) = 4$, $q > 4$, and either $r \ge \frac{3n}{4}$ or 8 divides $n$, then $S_q \cap S_r$ contains a usable unit. Similarly, if $\gcd(q,n) = 2$, $q > 2$, $r \ge \frac{3n}{4}$, and 4 divides $n$, then $S_q \cap S_r$ contains a usable unit.
\end{lem}

\begin{proof}
For the first part of the assertion: we write $q = 2^mb$, $n = 2^mc$, with $\gcd(b,c) = 1$. We let $d = b^{-1} \in \BZ/c$. Then $1 < d < c = 2^{-m}n$, and $\gcd(d,c)=1$ implies that $d$ is coprime to all factors of $n$ except possibly 2. We let $e = d$ if $d$ is odd and $e = d + 2^{-m} n$ otherwise; then $1 < e < 2^{-m+1}n$ is a unit mod $n$. (We have that $e$ is coprime to all factors of $n$ except possibly 2. If $d$ is odd, then $d=e$ is coprime also to 2, therefore to $n$. On the other hand, if $d$ is even, then $c = 2^{-m}n$ must be odd, so $d + 2^{-m}n$ is odd and still coprime to $2^{-m}n$.) Furthermore, the set $\{e + \frac{kn}{4}\}$ consists of units mod $n$, as $\frac{n}{4}$ and $n$ have the same prime divisors. And as $1 < e < \frac{n}{4}$, this set consists of usable units. Then, as $\gcd(q,r,n) = 1$, $r$ is odd, so $\{(e + \frac{kn}{4})r\} = \{er + \frac{ln}{4}\}$ has some element in $[0, \frac{n}{4}) \subset [0, 2r-n)$. This means that at least one of the elements of $\{e + \frac{kn}{4}\}$ is a usable unit in $S_q \cap S_r$.

For the assertion about $\gcd(q,n) = 4$ and $r \ge \frac{3n}{4}$ and the assertion about $\gcd(q,n) = 2$, the proof is the same, except that we consider the set $\{e + \frac{kn}{2}\}$, and $[0, \frac{n}{2}) \subset [0, 2r-n)$. (In the case $\gcd(q,n) = 2$, $d$ is odd because 4 divides $n$, so $c$ is even.) For the assertion about $\gcd(q,n) = 4$ and 8 divides $n$, $d$ must be odd, and we consider the set $\{d + \frac{kn}{4}\}$.
\end{proof}

\begin{lem}
\label{gcdfour}
Suppose $\gcd(q,n) = 4$ and neither of the two conditions in the previous lemma are satisfied (i.e., 8 does not divide $n$ and $\frac{2n}{3} < r \le \frac{3n}{4}$), and suppose $q > 16$. Then $S_q \cap S_r$ contains a usable unit.
\end{lem}

\begin{proof}
As before, we write $q = 4b$, $n = 4c$, with $\gcd(b,c) = 1$ and $c$ odd. We let $d = b^{-1} \in (\BZ/c)^\times$; then letting $y = d$ if $d$ is odd and $y = d + \frac{n}{4}$ otherwise, we have that $y, y + \frac{n}{2} \in (\Z/n)^\times$ by the Chinese remainder theorem. (These are coprime to $c$, therefore to all prime divisors of $n$ but 2, and we choose $y$ to be odd.) For the same reason,
$$\left\{y, y + \frac{n}{2}, 2y + \frac{n}{4}, 2y + \frac{3n}{4}, 4y + \frac{n}{4}, 4y + \frac{3n}{4}\right\}$$
are all units mod $n$. (Recall that $c$ is odd, so $2y, 4y \in (\BZ/c)^\times$.) Furthermore, if $q > 16$, these are all usable units, as unusable units are $\equiv 1$ mod $c$, but $b > 4$ is the smallest multiple of $y$ such that $by \equiv 1$ mod $c$.

We will see that one of these must be in $S_r$: First of all, as $r$ is odd, we have $y(r + \frac{n}{2}) \equiv yr + \frac{n}{2}$, and if neither of these is in $[0, 2r-n) \supset [0, \frac{n}{3}]$, we can assume $yr \in (\frac{n}{3}, \frac{n}{2})$. (Otherwise we switch $y$ and $y + \frac{n}{2}$.) Then $2yr \in (\frac{2n}{3}, n)$, and as $r$ is odd, $\{\frac{n}{4}r, \frac{3n}{4}r \} = \{\frac{n}{4}, \frac{3n}{4}\}$. Assuming without loss of generality that $\frac{n}{4}r \equiv \frac{n}{4}$, $yr \in (\frac{n}{3}, \frac{n}{2})$ implies $(2y + \frac{n}{4})r \in (\frac{11n}{12}, \frac{n}{4})$. (Otherwise we would pick $2y + \frac{3n}{4}$.) Then $2y + \frac{n}{4} \notin S_r$ would imply $(2y + \frac{n}{4})r \in (\frac{11n}{12}, n)$, i.e., $2y \in (\frac{2n}{3}, \frac{3n}{4})$. Repeating this argument, we assume without loss of generality that $\frac{3n}{4} r \equiv \frac{3n}{4}$, and then $2y \in (\frac{2n}{3}, \frac{3n}{4})$ implies $(4y + \frac{3n}{4})r \in (\frac{n}{12}, \frac{n}{4})$, which is within our target zone. So if the first four units listed above are not in $S_r$, then one of the last two is, and this one is a usable unit in $S_q \cap S_r$.
\end{proof}

We will end with a proposition about the case $\gcd(q,n) = 2$, whose strategy is similar to that of the previous proposition.

%

\begin{prop}
\label{notfourdividesn}
If $\gcd(q,n) = 2$, 4 does not divide $n$, and there is some $m \in \N$ such that $q > 2^m$ and $r \ge \frac{3n}{4} + \frac{n}{2^{m+2}}$, then $S_q \cap S_r$ contains a usable unit.
\end{prop}

\begin{proof}
As before, we write $q=2b, n = 2c$, with $\gcd(b,c) =1$ and $c$ odd. Let $d = b^{-1} \in \BZ/c$ and $y := d$ if $d$ is odd and $y := d + \frac{n}{2}$ otherwise. As $c$ is odd, $2^ky +\frac{n}{2}$ is a unit mod $n$ ($k \in \N$), and is in $S_q$ if $q > 2^k$.

First of all, if $yr \le \frac{n}{2}$, then $y \in S_q \cap S_r$ is the needed unit. (There is no issue of usability here, as $\frac{n}{2}+1$ is not a unit, and $1 \notin S_r$.) If this fails, we see if $(2y + \frac{n}{2})r \le \frac{n}{2}$. If these both fail, we have that $yr > \frac{n}{2}$, and as $\frac{n}{2}r \equiv \frac{n}{2}$, we have $2yr < \frac{n}{2}$; combining these two, we must have $yr \in (\frac{n}{2}, \frac{3n}{4})$. We continue to $4y + \frac{n}{2}$; if this is not in $S_r$, we have $4yr < \frac{n}{2}$, and combined with the previous conditions, we get $yr \in (\frac{n}{2}, \frac{5n}{8})$, etc. So continuing this process to $2^m (<q)$, we get that either there is some $l \le m$ with $(2^ly + \frac{n}{2})r \le \frac{n}{2}$ and so $2^ly + \frac{n}{2}$ is a usable unit in $S_q \cap S_r$, or $yr \in (\frac{n}{2}, \frac{(2^m+1)n}{2^{m+1}})$. Then if $r \ge \frac{3n}{4} + \frac{n}{2^{m+2}}$, $yr < 2r-n$ and so $y$ is a usable unit in $S_q \cap S_r$.
\end{proof}

\section{The case $q \le \frac{\sqrt{n}}{2}$}
In the first part of this section, we prove that $S_p \cap (S_q \cup S_r)$ contains a usable unit if $S_p$ does (in the case $q \le \frac{\sqrt{n}}{2}$). We recall from Proposition \ref{usableunits} that $S_p$ almost always contains a usable unit; however, there can be problems when $p = 1,2,4$, and these are dealt with in the second part of this section.

\begin{lem}
\label{qsmall}
If $S_p$ contains a usable unit and $q \le \frac{\sqrt{n}}{2}$, $S_p \cap S_r$ or $S_p \cap S_q$ contains a usable unit.
\end{lem}

\begin{proof}
First of all, we can assume $\gcd(p,q) = 1$, as otherwise $S_p \cap S_q$ contains the usable unit $\gcd(p,q)^{-1}$ (see Corollary \ref{factor}).

Suppose we have $a \in S_p \setminus S_r$. Then we have $0 \le [ap]_n < 2p$ and
$$n > [ar]_n \ge [2r]_n = n - 2p - 2q$$
Now, as $p+q+r = n$, $[x]_n < n$, and $[ar]_n > 0$, we have
$$[ap]_n + [aq]_n + [ar]_n = n \textrm{ or }2n$$
In the first case,
$$[ap]_n \ge 0, [ar]_n \ge n - 2p - 2q \implies [aq]_n \le 2p + 2q$$
So either $[aq]_n < 2q$, in which case $a \in S_q$, or $2q \le [aq]_n \le 2p + 2q < 4q$. (We assumed $\gcd(p,q) = 1$, so $p < q$.) Then $a - 2 \in S_q$.

In the second case,
$$[ap]_n < 2p, [ar]_n < n \implies [aq]_n \ge n-2p+2 \implies a+2 \in S_q$$
So we conclude that $a \in S_p \setminus S_r$ implies that one of $a, a \pm 2 \in S_q$.

Recalling the description of $S_p$ and $S_q$ in Observation \ref{origs}, if there are elements of $S_p$ and $S_q$ within distance 2 of each other, then there must be integers $0 \le b < p, 0 \le c < q$ such that
$$\left|  \ceil[\bigg]{\frac{nb}{p}} (+1) - \left( \ceil[\bigg]{\frac{nc}{q}} (+1)\right) \right| \le 2 \implies |bq - cp| < 4\frac{pq}{n}$$
If $q \le \frac{\sqrt{n}}{2}$, then $4\frac{pq}{n} \le 1$, so as $b,c,p,q$ are integers, the only way this can happen is if $bq = cp$, and as we assumed $\gcd(p,q) = 1$, this only can happen for $b=c=0$. So the only elements of $S_p$ and $S_q$ within distance 2 of each other are $0,1 (\in S_p \cap S_q)$. Then by the previous paragraph, all elements of $S_p$ other than $\{0,1\}$ must be in $S_r$. So if $S_p$ contains a usable unit, this usable unit is in $S_p \cap S_r$.
\end{proof}

\begin{prop}
\label{qsmallpbad}
If $S_p$ does not contain a usable unit and $q \le \frac{\sqrt{n}}{2}$ and $r > \frac{2n}{3}$ and $n \ge 30$, then either $(p,q,r)$ belongs to one of families 1, 2, 3 in Theorem \ref{maintwo} (in which case the \cite{MW} criterion is not satisfied) or $S_q \cap S_r$ contains a unit.
\end{prop}

\begin{proof}
We recall from Proposition \ref{usableunits} that $S_p$ does not contain usable units exactly when: $p = 1$; $p = 2$ and $n$ is even; or $p = 4$ and $n \equiv 4$ mod 8.

We will start by considering the case $\gcd(q,n) = 1$ and $q > 1$. 
We write $[q^{-1}]_n = \frac{kn}{p} + m$ for some $0 \le k < p$ and $0 \le m < \frac{n}{p}$. (In all of these cases, note that $p$ divides $n$.) Letting $l = [kq]_p$, we have
 $\frac{kn}{p}q \equiv \frac{ln}{p}$ mod $n$, and we claim that $m > \frac{n}{pq}$, as otherwise $mq < \frac{n}{p}$ (equality not being possible since $\gcd(q,n) = 1$) and so, if $l > 0$,
$$n \ge \frac{(l+1)n}{p} > \frac{ln}{p} + mq > \frac{ln}{p} > 1 \implies \left[\left(\frac{kn}{p}+m\right)q\right]_n\ne 1$$
and if $l = 0$, then $[mq]_n = mq > 1$.

Now,
$$q^{-1}r \equiv q^{-1}(-p-q) \equiv -p\left(\frac{kn}{p} + m\right) - 1 \equiv -mp - 1$$
We claim $[q^{-1}r]_n = n -mp - 1$: first of all, $m < \frac{n}{p}$, so $mp < n$, and if $p > 1$, then $mp + 1 < n$ also, as $n$ is also a multiple of $p$. On the other hand, if $p = 1$, then $m + 1 < n$ unless $q^{-1} = m = n-1$, but this would imply $q = n-1 > \frac{n}{2}$. Now, using that $q \le \frac{\sqrt{n}}{2}$,
$$mp + 1 > \frac{n}{q} + 1 \ge 4q +1 > 2p + 2q \implies [q^{-1}r]_n <  [2r]_n,$$
so $q^{-1} \in S_q \cap S_r$ is a usable unit.

At this point, we just need to put together what we have already proved. If $p = 2,4$, then the above gives a proof for $\gcd(q,n) = 1$ (automatically $q > 1$), and as $\gcd(p,q,n) = 1$, the only other option is that $\gcd(q,n) > 1$ has some prime factor other than 2, in which case Lemma \ref{gcdqnnottwo} applies. So the only case to consider is $p = 1$. If $\gcd(q,n)= 1$ and $q > 1$, then we use the above; if $q = 1$, then this is family 1 in Theorem \ref{maintwo}. If $\gcd(q,n) > 1$ is not a power of two, then Lemma \ref{gcdqnnottwo} applies, and if $\gcd(q,n)$ is a power of two $\ge 16$, then Lemma \ref{poweroftwo} applies. So it remains to consider the cases $\gcd(q,n) = 2,4,8$ (and $p = 1$).
\begin{itemize}
\item $\gcd(q,n) = 2$:
\begin{itemize}
\item $q = 2$: This is family 2 in Theorem \ref{maintwo}.
\item $q > 2$ and 4 divides $n$: $r \ge n - \frac{\sqrt{n}}{2} -1 > \frac{3n}{4}$ for $n \ge 11$, so by Lemma \ref{poweroftwo}, $S_q \cap S_r$ contains a usable unit.
\item $q > 2$ and 4 does not divide $n$: $r \ge n - \frac{\sqrt{n}}{2} - 1 > \frac{7n}{8}$ for $n \ge 30$, so by Proposition \ref{notfourdividesn} (applied with $m = 1$), $S_q \cap S_r$ contains a usable unit.
\end{itemize}
\item $\gcd(q,n) = 4$:
\begin{itemize}
\item $q = 4$ and 8 does not divide $n$: This is family 3 in Theorem \ref{maintwo}.
\item $q = 4$ and 8 does divide $n$: One of $\frac{n}{4}+1, \frac{3n}{4}+1$ is a usable unit in $S_q \cap S_r$ if $r \ge \frac{3n}{4}$ (which is true, by the above, for $n \ge 11$). 
\item $q > 4$: $r > \frac{3n}{4}$ for $n \ge 11$, so by Lemma \ref{poweroftwo}, $S_q \cap S_r$ contains a usable unit.
\end{itemize}
\item $\gcd(q,n) = 8$:
\begin{itemize}
\item $q = 8$: One of $\frac{n}{4}+1, \frac{3n}{4}+1$ is a usable unit in $S_q \cap S_r$ if $r \ge \frac{3n}{4}$ (in particular, for $n \ge 11$).
\item $q > 8$: By Lemma \ref{poweroftwo}, $S_q \cap S_r$ contains a usable unit.
\end{itemize}
\end{itemize}
\end{proof}
At this point, it might be worth noting in families 1--3 in Theorem \ref{maintwo}, $S_p$ and $S_q$ both do not contain usable units, by Proposition \ref{usableunits}, so that the \cite{MW} criterion is definitely not satisfied. (Indeed, families 1 and 2 are the known families of lattice triangles.)

\section{The case $q > \frac{\sqrt{n}}{2}$, part 1}
The main technique in this case is a rational approximation argument, which will be started at the end of this section and finished in \S8. We will begin by eliminating the cases in which this argument cannot be used. (This includes the case $\gcd(q,n) > 2$, where we recall results from \S5, as well as the case $r = 3q + p$, which turns out to be problematic.) The second half of this section introduces the number theoretic background necessary for the argument (in particular, the Jacobsthal function) and deals with the cases in which the rational approximation has denominator $> 3$; when the denominator is $\le 3$, one needs to use a slightly different strategy (depending on fairly fine case distinctions), which will be the topic of \S8.

\begin{lem}
\label{qgcdbig}
Suppose $q > \frac{\sqrt{n}}{2}$, $r > \frac{2n}{3}$, $\gcd(q,n) > 2$, and $n \ge 1024$. Then $S_q \cap S_r$ contains a usable unit.
\end{lem}

\begin{proof}
If $\gcd(q,n)$ is not a power of two, then this follows by Lemma \ref{gcdqnnottwo}. If $\gcd(q,n)$ is a power of two $\ge 8$ and $n \ge 256$, then $q > 8$, so this follows by Lemma \ref{poweroftwo}. If $\gcd(q,n) = 4$, then $n \ge 1024$ implies $q > 16$, so one of Lemma \ref{poweroftwo} or Lemma \ref{gcdfour} applies.
\end{proof}

\begin{prop}
\label{threeq}
If $r = 3q + p$, then the \cite{MW} criterion is applicable iff $p \ne 1,2,4$.
\end{prop}

\begin{proof}
First of all, as $\gcd(p,q,r) = 1$, we must have $\gcd(p,q) = 1$, and then since $n = 4q + 2p$, we must have $\gcd(p,n) = 1,2,4$.

We write $p = ab, n = ac$, with $\gcd(b,c) = 1$. Let $d = b^{-1} \in (\Z/c)^\times$ if $b^{-1}$ is odd, and $d = b^{-1} + c$ otherwise. Then $d \in (\Z/n)^\times$ is a usable unit if $b \ne 1$ (i.e., $p \ne a$), and
$$0 \equiv dn = d(4q+2p) \implies 4dq \equiv -2a$$
Then $[dq]_n = \frac{kn}{4} - \frac{a}{2}$ for some $k = 1,2,3,4$, and $dr \equiv \frac{3kn}{4} - \frac{a}{2}$ mod $n$. If $k = 1$, $dq < \frac{n}{4} < 2q$, so $d \in S_p \cap S_q$. If $k = 3$, $dr = \frac{n}{4} - \frac{a}{2} <  2q = [2r]_n$, so $d \in S_p \cap S_r$. We claim that $k \ne 2,4$: 
if $a = 1$, then as $n$ is even, $k = 2,4$ would imply $dq$ has nonzero fractional part; 
if $a = 2$, then $k = 2,4$ would imply $dq \equiv -1$ mod $\frac{n}{2}$, which, given that $q < \frac{n}{2}$, implies $q = \frac{n}{2} - \frac{p}{2}$, which is impossible since $q = \frac{n}{4} - \frac{p}{2}$; 
and if $a = 4$, $k = 2,4$ would imply $dq$ is even, but $d,q$ are odd.

So we have shown that if $p \ne a = 1,2,4$, then there is a usable unit in $S_p \cap S_q$ or $S_p \cap S_r$. It remains to see what happens if $p = 1,2,4$. First of all, as $n = 4q+2p$ (with $q$ odd if $p$ is even), these are exactly the cases where $S_p$ does not contain a usable unit, by Proposition \ref{usableunits}. So one it suffices to check that $S_q \cap S_r$ does not contain a usable unit. And as $n = 4q + 2p$ and $\gcd(p,q) = 1$, we must have $\gcd(q,n) = 1,2$.

If $\gcd(q,n) = 1$, then
$$0 \equiv q^{-1}(4q+2p) = 4 + 2q^{-1}p \implies q^{-1}p \equiv -2 \mod \frac{n}{2}$$
In particular, $q^{-1}r \equiv 3 + q^{-1}p \equiv 1$ mod $\frac{n}{2}$, and as $q \ne r$, this must mean $q^{-1}r = \frac{n}{2}+1$, so for $k < \frac{n}{2}$, we have
$$[kq^{-1}r]_n = \begin{cases} k & k \textrm{ even } \\ \frac{n}{2}+k & k \textrm{ odd} \end{cases}$$
and as $[2r]_n = 2q < 2q + p = \frac{n}{2}$ and $S_q = \{kq^{-1}: 0 \le k < 2q\}$, $S_q \cap S_r$ consists entirely of even numbers, i.e., does not contain a unit. 

Similarly, if $\gcd(q,n) = 2$, then writing $q = 2b, n = 2c, d = b^{-1} \in (\BZ/c)^\times$ if $b^{-1}$ is odd and $b^{-1} + \frac{n}{2}$ otherwise, we have
$$S_q = \{ kd + \frac{ln}{2}: 0 \le k < q, 0 \le l < 2\}$$
and
$$dr = d\left(\frac{n}{2} + q\right) \equiv \frac{n}{2} + 2$$
(as $d$ is odd). Then, as $r$ also is odd, for $k < q < \frac{n}{4}$,
$$[(kd + \frac{ln}{2})r]_n = \begin{cases} 2k & k + l \equiv 0 \mod 2 \\ 2k + \frac{n}{2} & k + l \equiv 1 \mod 2 \end{cases}$$
so $[2r]_n = 2q < \frac{n}{2}$ implies that $S_q \cap S_r$ can only have elements $kd + \frac{ln}{2}$ with $k + l$ even. But as $p$ is odd, $\frac{n}{2} = 2q + p$ is odd, and $d$ is odd, so this means $S_q \cap S_r$ consists entirely of even numbers and therefore does not contain a unit.
\end{proof}
So we have established that families 4--6 in Theorem \ref{maintwo} are indeed families for which the \cite{MW} condition is not satisfied. Now, before stating the next lemma, we must introduce a new function:

\begin{defn}
\label{jacobsthal}
The \emph{Jacobsthal function} $j(n)$ is defined to be the smallest integer $m$ such that any sequence of $m$ consecutive integers must contain a number coprime to $n$.
\end{defn}
This function was introduced in \cite{Jacobsthal} and will be our main tool to prove the existence of units in $S_q \cap S_r$. We will need a few facts about $j(n)$ first.

\begin{defn}
We will call an arithmetic progression $\{a + xb: x \in \Z\}$ mod $n$ a ``good" progression if $\gcd(a,b,n) = 1$. (The gcd condition exactly ensures that a sequence of $j(n)$ consecutive terms includes a number coprime to $n$.)
\end{defn}

Our goal will be to find an arithmetic progression of length $j(n)$ in $S_q \cap S_r$, provided that $n$ is sufficiently large. For the purposes of the following lemma, it will suffice to establish a preliminary bound on $j(n)$, though we will need somewhat more in \S9.

\begin{fact}
\label{Kan}
\cite[Satz 4]{Kanold} Let $\omega(n)$ be the number of distinct prime factors of $n$. Then
$$j(n) \le 2^{\omega(n)}$$
\end{fact}

\begin{fact}
\label{Rob}
\cite[Th\'eor\`eme 11]{Robin} For $n \ge 3$,
$$\omega(n) \le 1.3841 \frac{\ln n}{\ln \ln n}$$
\end{fact}

\begin{rem}
\label{jnnbound}
For the purposes of the next lemma, we will establish a bound of the form $j(n) < cn$ for $n > 10000$: Combining Facts \ref{Kan} and \ref{Rob},
$$j(n) \le 2^{\omega(n)} < e^{\ln 2 \times 1.39 \ln n/\ln \ln n} < e^{.97 \ln n/\ln \ln n}$$
The function $f(x) = x^{\frac{.97}{\ln \ln x} - 1}$ has negative derivative for $x > e$, so for $n > 10000$ we have
$$j(n)/n < f(n) < f(10000) < .006$$
\end{rem}

\begin{lem}
\label{rationalapprox}
Suppose $\gcd(q,n) = 1,2$, $q > \frac{\sqrt{n}}{2}$, $r > \frac{2n}{3}$, and $n > 10000$. If $24j(n)^2 < \sqrt{n}$, then one of the intersections $S_p \cap S_q, S_p \cap S_r, S_q \cap S_r$ contains a usable unit unless $p = 1,2,4$ and $r = 3q+p$ (see Proposition \ref{threeq}).
\end{lem}

\begin{proof}
If $\gcd(q,n) = 1$, recall $S_q = \{kq^{-1}: 0 \le k < 2q\}$. We note that $S_q \cap S_r$ does not contain unusable units, because $1 \notin S_r$, and if $\frac{n}{2}+1$ is a unit, $n$ is even, so $q$ is odd, and then $[(\frac{n}{2}+1)q]_n = \frac{n}{2} + q > 2q$, so $\frac{n}{2}+1 \notin S_q$.

If $\gcd(q,n) = 2$, as usual, we write $q = 2b, n = 2c$ and let $d = b^{-1} \in \Z/c$. We can choose $z = d$ or $d + \frac{n}{2}$ such that $z$ is a unit and such that the subset $S_q' :=\{kz: 0 \le k < q\} \subset S_q$ contains no unusable unit but 1. (If $c$ is odd, then $\frac{n}{2}+1$ is not a unit, so we choose $z$ to be whichever one of these is odd. If $c$ is even, $d$ and $d + \frac{n}{2}$ are both units, and we choose the one which is $b^{-1}$ mod $n$.) Then $S_q' \cap S_r$ does not contain unusable units.

We let $x = [q^{-1}r]_n/n$ or $[zr]_n/n$, depending on whether $\gcd(q,n) = 1$ or $2$. By hypothesis, we can choose some $N \in (12 j(n), \frac{\sqrt{n}}{2j(n)})$. Then Dirichlet's approximation theorem states that there are relatively prime integers $\alpha, \beta$ with $1 \le \beta \le N$ such that
$$\left| x - \frac{\alpha}{\beta} \right| < \frac{1}{\beta N}$$

Suppose there is some $\gamma \in (\Z/\beta)^\times$ with $\{\frac{\gamma\alpha}{\beta}\} \in [\frac{1}{9}, \frac{2}{9}]$. Then, for $k < j(n)$,
$$\left|(\gamma + k\beta)\left(x - \frac{\alpha}{\beta}\right)\right| < (\gamma + k\beta) \frac{1}{\beta N} < \frac{j(n)}{N} < \frac{1}{12}$$
So $(\gamma + k\beta)q^{-1}r$ or $(\gamma + k\beta)zr$ is within $\frac{n}{12}$ of $\frac{\gamma\alpha}{\beta} n \in [\frac{n}{9}, \frac{2}{9}]$, and hence in $[0, \frac{n}{3}] \subset [0, 2r-n)$. This means that
$$\{(\gamma + k\beta)q^{-1}: 0 \le k \le j(n)-1\} \subset S_r$$
(or the same when $q^{-1}$ is replaced by $z$). And, for $k < j(n)$,
$$\gamma + k\beta < j(n)N < \frac{\sqrt{n}}{2} < q \implies \{(\gamma + k\beta)q^{-1}: 0 \le k \le j(n)-1\} \subset S_q$$
(or the same for $z$ and $S_q'$). Then $S_q \cap S_r$ or $S_q' \cap S_r$ contains a good progression of length $j(n)$, and hence a usable unit.

The question is now if such a $\gamma$ exists. If $\beta > 10000$, Remark \ref{jnnbound} implies that there is an element $\delta$ of $(\Z/\beta)^\times$ in $[\frac{\beta}{9}, \frac{2\beta}{9}]$, so $\alpha^{-1}\delta$ is such a $\gamma$. For $\beta \in [4, 10000]$, a computer search reveals that the only values of $\beta$ for which this is not the case are $\beta = 4,10,18,30$. For $\beta = 4$, we take $\gamma = \alpha^{-1}$, so we are ``aiming for" $\frac{1}{4}$ and are permitted an error $< \frac{1}{12}$; for $\beta = 10$, we take $\gamma = \alpha^{-1}$ if $x > \frac{\alpha}{\beta}$, allowing an error of $<\frac{7}{30}$, and $\gamma = 3\alpha^{-1}$ if $x \le \frac{\alpha}{\beta}$, allowing an error of $< \frac{3}{10}$; for $\beta = 18$ we take $\gamma = \alpha^{-1}$ or $5\alpha^{-1}$ (for $x > \frac{\alpha}{\beta}$ or $x \le \frac{\alpha}{\beta}$, resp.) and are allowed an error of $< \frac{5}{18}$; and for $\beta = 30$, $\gamma = \alpha^{-1}$ or $7\alpha^{-1}$, allowing an error of $< \frac{7}{30}$. So the bound $12j(n) < N$ is sufficient, as this gives an accumulated error $< \frac{j(n)}{N} < \frac{1}{12}$.

For $\beta \le 3$, it is clearly impossible to aim for a unit in $(0, \frac{1}{3})$, so we will need a slightly more sophisticated method of choosing a good progression in $S_q \cap S_r$. It is for that reason that we separately treat each case $x \le \frac{\alpha}{\beta}$ or $ x> \frac{\alpha}{\beta}$ for each $\alpha \le \beta \le 3$ in \S8.
\end{proof}

\section{The case $q > \frac{\sqrt{n}}{2}$, part 2}
What follows is a list of propositions explaining what happens when $x$ is over- or under-approximated by a given fraction of denominator $\le 3$. The proof strategies are very similar in each proposition: there is a certain balancing act involved, as one identifies a good arithmetic progression of length $\ge j(n)$ which is in $S_r$ because sufficient error has built up that $x$ is very far from its approximation, but on the other hand, one is not allowed to wait too long for the error to build up, as multiples of $q^{-1}/z$ may no longer be in $S_q$. As each case is somewhat different in terms of minimal size of error, needed amount of built-up error, size of $q$, and resulting necessary bounds on $n$, it has not been possible to condense these in any useful way. (To be clear, each proposition is a proof of Lemma \ref{rationalapprox} in the case described in the statement of the proposition.) The bounds obtained in these propositions will also be used in \S9 to determine what needs to be checked by computer, as this is where the reduction algorithm described there is least useful. (We do not claim that these bounds are optimal, as they are not, but they will be sufficient to reduce the needed computation to checking only triples with $n \le 10000$.)

\begin{prop}
The case $\alpha = 0$ and the case $\alpha = 1, \beta = 3, x \le \frac{\alpha}{\beta}$.
\end{prop}

\begin{proof}
As $r > \frac{2n}{3}$, $[0, \frac{n}{3}] \subset [0, 2r-n)$, so in any of these cases $q^{-1}/z$ is a usable unit in $S_q \cap S_r$.
\end{proof}

\begin{prop}
The case $\alpha = 1, \beta = 3, x > \frac{\alpha}{\beta}$. In this case, we use that $j(n) < \frac{n}{216}$.
\end{prop}

\begin{proof}
We write $q^{-1}r/zr = \frac{n}{3} + m$, with $0 < m < \frac{n}{3N}$. First of all, if $r \ge \frac{3n}{4}$, $2r - n \ge \frac{n}{2} > x$, so $q^{-1}/z$ is a usable unit in $S_q \cap S_r$, so we may assume $r < \frac{3n}{4}$. And if $\gcd(3,n) = 1$, then $3q^{-1}/3z$ is a usable unit in $S_q \cap S_r$, so we may assume $m \in \Z$. Also, in this case, $r > \frac{2n}{3}$ implies $2r-n \ge \frac{n}{3} + 2$, so if $m = 1$, then $q^{-1}/z \in S_q \cap S_r$. We therefore assume $m \ge 2 \in \Z$. At this point, it will be convenient to consider the $\gcd(q,n) = 1$ and 2 cases separately.

In the $\gcd(q,n) = 1$ case, $m \ge 2$, $r < \frac{3n}{4}$, and
$$(2q - 3j(n))m > 2\left(\frac{n}{4} - 3j(n)\right) > \frac{n}{3}$$
when $6j(n) < \frac{n}{6}$. Then, if we let $l$ be the smallest positive integer such that $(2+3l)m > \frac{n}{3}$, we must have
$2+3l \le 2q - 3j(n) + 2$. Furthermore, $3j(n)m < \frac{j(n)}{N} n < \frac{n}{3}$, and so
$$\{(2+3k)q^{-1}: l \le k \le l + j(n) - 1\} \subset S_q \cap S_r$$
as $[(2+3k)q^{-1}r]_n = [\frac{2n}{3} + (2+3k)m]_n \in [0, \frac{n}{3})$ and $2+3k < 2q$ for $k$ in this range. So $S_q \cap S_r$ contains a good progression of length $j(n)$, and therefore a usable unit. (This is the prototype for the arguments that will be used throughout this section.)

In the $\gcd(q,n) = 2$ case, we will need to deal with the case $m = 2$ separately before applying an argument as above. In this case,
$$zr \equiv \frac{n}{3} + 2 \implies 2r \equiv q \frac{n}{3} + 2q \implies 2r \equiv 2q \mod \frac{n}{3} \implies r \equiv q \mod \frac{n}{6}$$
Assuming $\frac{2n}{3} < r < \frac{3n}{4}$ (so also $\frac{n}{8} < q < \frac{n}{3}$), this can only happen if $r = \frac{n}{2} + q$, which we recognize as the case $r = 3q + p$, which is dealt with in Proposition \ref{threeq}. So we may assume $m \ge 3$. Then
$$(q - 3j(n))m \ge 3\left(\frac{n}{8} - 3j(n)\right) > \frac{n}{3}$$
when $9j(n) < \frac{n}{24}$. Then by the same argument as before, $S_q' \cap S_r$ contains a good progression of the form $(2+3k)q^{-1}$ of length $j(n)$, and therefore a usable unit.

\end{proof}

\begin{prop}
\label{eightthree}
The case $\alpha = 1, \beta = 2, x \le \frac{\alpha}{\beta}$. In this case, we use $j(n) < \frac{n}{24}$.
\end{prop}

\begin{proof}
As before, we write $q^{-1}r/zr = \frac{n}{2} - m$ with $0 \le m < \frac{n}{2N}$. If $r > \frac{3n}{4}$, then $q^{-1}/z \in S_q \cap S_r$ is a usable unit, so we may assume $r \le \frac{3n}{4}$. To establish a lower bound on $m$, we again split into cases based on $\gcd(q,n)$.

If $\gcd(q,n) = 1$, then $r \ne \frac{n}{2}$ implies that $m \ne 0$. Also $m \ne \frac{1}{2}$, as otherwise we would have $2q^{-1}r \equiv -1$ and $2r-n = n - q \ne n - 2p -2q$. Now we apply the same argument as before:
$$(2q - 2j(n))m \ge \frac{n}{4} - 2j(n) > \frac{n}{6}$$
for $2j(n) < \frac{n}{12}$, so if $l$ is the smallest positive integer so that $(1+2l)m > \frac{n}{6}$, we have that $1+2l \le 2q - 2j(n) + 1$. Then $2j(n)m < \frac{j(n)}{N} n < \frac{n}{3}$, so as before,
$$\{ (1+2k)q^{-1}: l \le k \le l + j(n) - 1\} \subset S_q \cap S_r$$
gives a good progression of length $j(n)$.

If $\gcd(q,n) = 2$, recall that $z \equiv (\frac{q}{2})^{-1}$ mod $\frac{n}{2}$, so $zr \equiv -m$ mod $\frac{n}{2}$ implies $r \equiv -m\frac{q}{2}$ mod $\frac{n}{2}$. Given that $r \in (\frac{2n}{3}, n-q)$ and $q \in (0, \frac{n}{3})$, we can immediately rule out $m \le 2$. If $m = 3$, we must have $r = n - \frac{3}{2} q$, so $p = \frac{q}{2}$, and then by Proposition \ref{usableunits}, $S_p = S_p \cap S_q$ must contain a usable unit. Similarly, if $m = 4$, then we must have $r = n - 2q$, i.e., $p = q$, violating the condition $\gcd(p,q,n) = 1$. 
So it suffices to consider the case $m \ge 5$. In this case
$$(q - 2j(n))m \ge 5\left( \frac{n}{8} - 2j(n) \right) > \frac{n}{6}$$
as long as $10j(n) < \frac{11n}{24}$, so by the same argument as in the previous paragraph, $S_q' \cap S_r$ contains a good progression of the form $(1+2k)z$ of length $j(n)$, and therefore a usable unit.
\end{proof}

\begin{prop}
\label{eightfour}
The case $\alpha = 1, \beta = 2, x > \frac{\alpha}{\beta}$. In this case, we use $j(n) < \frac{n}{60}$.
\end{prop}

\begin{proof}
We have $q^{-1}r/zr = \frac{n}{2} + m$, $0 < m < \frac{n}{2N}$. If $r \ge \frac{4n}{5}$, then $q^{-1}/z \in S_q \cap S_r$ is a usable unit, so we may assume $r < \frac{4n}{5}$. Also, if $n$ is odd, $2q^{-1} \in S_q \cap S_r$ is a usable unit, so we may assume $n$ is even, and so $m \in \Z$. As usual, we break into cases at this point.

Suppose $\gcd(q,n)= 1$. Then $q^{-1}r = \frac{n}{2} + 1$ implies $r = 3q + p$, which is covered by Proposition \ref{threeq}. The case $m = 2$ will require more work: first of all, this implies $n = 6q + 2p, r = 5q + p$, so $2r -n > \frac{n}{2} + 2$ and $q^{-1} \in S_q \cap S_r$ unless $q \le p + 2$.
\begin{itemize}
\item $q = p$: $S_p = S_p \cap S_q$ contains a usable unit by Proposition \ref{usableunits}. 
\item $q = p+1$: $n = 8p + 6$, and $q$ is odd, so $p$ must be even. If 3 divides $p$, then $\gcd(p,n) = 6$, and so $S_p \cap S_r$ contains a usable unit by Lemma \ref{gcdqnnottwo}. Otherwise, we have $\gcd(p,n) = 2$. We write $p = 2a, n = 2c$ and choose $d$ to be coprime to $n$ and $\equiv a^{-1}$ mod $c$. Then
$$d(4p+3) \equiv 0 \mod \frac{n}{2} \implies d = -\frac{8}{3} + \frac{kn}{6}$$
for some $k = 1,\ldots, 6$. As $d$ is an integer, we rule out $k = 3,6$ (since then $d$ would have fractional part $\frac{1}{3}$), and as $d$ is odd, we additionally rule out $k = 2,4$ (since then $\frac{n-8}{3}, \frac{2n-8}{3}$ are even if they are integers). If $k = 1$,
$$dq = d(p+1) \equiv 2 + d = \frac{n}{6} - \frac{2}{3} < 2q \implies d \in S_p \cap S_q.$$
And if $k = 5$,
$$dr = d(6p + 5) \equiv 12 + 5d \equiv \frac{n}{6} - \frac{4}{3} < 2r-n \implies d \in S_p \cap S_r.$$
So one of $S_p \cap S_q$ and $S_p \cap S_r$ contains the usable unit $d$.
\item $q = p+2$: This case is very similar to the previous one, but now $n = 8p + 12$ and $p$ is odd, so if 3 divides $p$, $\gcd(p,n) = 3$, and $S_p \cap S_r$ contains a usable unit by Lemma \ref{gcdqnnottwo}. Otherwise $\gcd(p,n) = 1$, and $p^{-1}(8p+12) \equiv 0$ mod $n$ implies $p^{-1} = -\frac{2}{3} + \frac{kn}{12}$ for some $k = 1, \ldots, 12$. As $p^{-1}$ must be an odd integer, $k$ cannot be divisible by 2 or 3, so the possible values of $k$ are $1,5,7,11$. If $k = 1,7$,
$$p^{-1}q = p^{-1}(p+2) \equiv \frac{n}{6} - \frac{1}{3} < 2q \implies p^{-1} \in S_p \cap S_q.$$
If $k = 5,11$,
$$p^{-1}r = p^{-1}(6p+10) \equiv \frac{n}{6} - \frac{2}{3} < 2r - n \implies p^{-1} \in S_p \cap S_r.$$
\end{itemize}

So we can assume $m \ge 3$ (still in the case $\gcd(q,n) = 1$). Now
$$(2q - 2j(n))m > 3\left(\frac{n}{5} - 2j(n)\right) > \frac{n}{2}$$
if $6j(n) < \frac{n}{10}$. By the usual argument, we have a good sequence of the form $(1+2k)q^{-1}$ of length $j(n)$ in $S_q \cap S_r$ and hence a usable unit.

Now suppose $\gcd(q,n) = 2$. First of all, if $\frac{n}{2}$ is even, then $z + \frac{n}{2} \in S_q \cap S_r$ is a usable unit. On the other hand, if $\frac{n}{2}$ is odd, $2z + \frac{n}{2}$ is also a usable unit, as it is odd and relatively prime to $\frac{n}{2}$ (and $2 < \frac{q}{2} = z^{-1}$ mod $\frac{n}{2}$, so it is usable), and
$$(2z + \frac{n}{2})r \equiv 2zr + \frac{n}{2} \equiv 2m < \frac{n}{N} < \frac{n}{3}$$
so $2z + \frac{n}{2} \in S_q \cap S_r$ is a usable unit.
\end{proof}

\begin{prop}
The case $\alpha = 2, \beta = 3, x \le \frac{\alpha}{\beta}$. In this case, we use $j(n) < \frac{\sqrt{n}}{6}$.
\end{prop}

\begin{proof}
As usual, we write $q^{-1}r/zr = \frac{2n}{3} - m$ for some $0 \le m < \frac{n}{3N}$. In this case, we will show
$$\{(2+3k)q^{-1}/z: 0 \le k \le j(n) - 1\} \subset S_q \cap S_r.$$
As $3j(n) < \frac{\sqrt{n}}{2} < q$, this is in $S_q$. And $(2+3k)q^{-1}r/zr \equiv \frac{n}{3} - (2+3k)m$, so for these to be in $S_r$, it suffices that $3j(n) m < \frac{n}{3}$. So we have a good progression of length $j(n)$ in $S_q \cap S_r$.
\end{proof}

\begin{prop}
The case $\alpha = 2, \beta = 3, x > \frac{2}{3}$. In this case, we use $j(n) < \frac{\sqrt{n}}{6}$ and $j(n) < \frac{n}{24}$.
\end{prop}

\begin{proof}
We write $q^{-1}r/zr = \frac{2n}{3} + m$ for some $0 < m < \frac{n}{3N}$. If 3 does not divide $n$, then $3q^{-1}/3z$ is a usable unit in $S_q \cap S_r$, so we can assume 3 divides $n$, and so $m \in \Z$. We start with $\gcd(q,n) = 1$, splitting into cases based on the size of $r$.

If $r \ge \frac{17n}{24}$, we will show that $\{(2+3k)q^{-1}/z: 0 \le k \le j(n)-1\}$ is in $S_q \cap S_r$. For $S_q$, it suffices that
$3j(n) < q$. For $S_r$, we have
$$(2+3k)q^{-1}r/zr \equiv \frac{n}{3} + (2+3k)m < \frac{n}{3} + 3j(n)m < \frac{5n}{12} \le 2r - n$$
as $N > 12j(n)$. So $S_q \cap S_r$ has a good progression of length $j(n)$.

If $r < \frac{17n}{24}$ (and so $q > \frac{7n}{48}$), we split into cases based on $\gcd(q,n)$. If $\gcd(q,n) = 1$, $m = 1$ implies $r \equiv q$ mod $\frac{n}{3}$, and $q < \frac{n}{3}$ implies $r = \frac{2n}{3} + q = 5q + 2p > \frac{3n}{4}$, so in this case it suffices to consider $m \ge 2$. Then
$$(2q - 3j(n))m > 2 \left( \frac{7n}{24} - 3j(n)\right) > \frac{n}{3}$$
if $6j(n) < \frac{n}{4}$, and $3j(n) m < \frac{n}{3}$, so by the usual argument, there is a good progression of the form $(1+3k)q^{-1}$ of length $j(n)$ in $S_q \cap S_r$.

Now, suppose $\gcd(q,n) = 2$. If $\frac{n}{2}$ is even, $z + \frac{n}{2}$ is also a usable unit in $S_q$, which is also in $S_r$ as
$$(z + \frac{n}{2})r \equiv zr + \frac{n}{2} \equiv \frac{n}{6} + m < \frac{n}{6} + \frac{n}{3N} < \frac{n}{3}$$
On the other hand, if $\frac{n}{2}$ is odd, $4z + \frac{n}{2}$ is a usable unit in $S_q$, which is also in $S_r$ as
$$(4z + \frac{n}{2})r \equiv 4zr + \frac{n}{2} \equiv \frac{n}{6} + 4m < \frac{n}{6} + \frac{4n}{3N} < \frac{n}{3}$$
(since $N > 12j(n) \ge 24$).
\end{proof}

\begin{prop}
\label{eightseven}
The case $\alpha = \beta$. In this case, we use $j(n) < \frac{3\sqrt{n}}{14}$.
\end{prop}

\begin{proof}
We write $q^{-1}r/zr = n - m$ for some $0 < m < \frac{n}{N}$.

Suppose $\gcd(q,n) = 1$. Then $p+q \equiv mq$, so $m = 1$ is impossible, and $m = 2$ implies $p = q$, so $S_q = S_p \cap S_q$ contains the usable unit $q^{-1}$. Similarly, $m = 3$ is impossible, as $p+q \le 2q < 3q < n$. So we have
$$(2q - j(n))m \ge 4(2q - j(n)) > 4q \ge 2p + 2q$$
if $j(n) < q$. Then if $l$ is the smallest positive integer such that $[lq^{-1}r]_n < [2r]_n = n-2p-2q$, we have $l \le 2q - j(n)$, and $j(n) m < \frac{n}{3}$, so by the usual argument there is a good progression $\{kq^{-1}: l \le k \le l + j(n) - 1\} \subset S_q \cap S_r$.

Suppose $\gcd(q,n) = 2$. Then $zr$ is odd, so $m$ must be odd. By the same argument as in the proof of Proposition \ref{eightthree}, we must have $m \ge 5$. Furthermore $m \ne 5$, as this would imply $p \equiv \frac{3q}{2}$ mod $\frac{n}{2}$, but $p \le q$ and $q < \frac{n}{3}$ make this impossible. So $m \ge 7$, and
$$(q - j(n))m \ge 7q - 7j(n) > 4q \ge 2p + 2q$$
if $7j(n) < 3q$. Then, as in the previous paragraph, there is a good progression of the form $kq^{-1}$ of length $j(n)$ in $S_q' \cap S_r$.
\end{proof}

The most restrictive bounds needed in these propositions were $j(n) < \frac{\sqrt{n}}{6}$ and $j(n) < \frac{n}{216}$. These are certainly satisfied under the conditions of Lemma \ref{rationalapprox}, but this information will be useful in the next section.

\section{Computer verification}
At this point, we have proved Theorem \ref{maintwo} up to finitely many exceptions. We will see what remains to be checked, give a computer-assisted proof that vastly less must actually be checked, and then present the results of the checking. (We start with the assumption that we will check all triples with $n \le 10000$.)

Recall that we split the proof into major cases $q \le \frac{\sqrt{n}}{2}$ and $q > \frac{\sqrt{n}}{2}$. When $q \le \frac{\sqrt{n}}{2}$, Lemmas \ref{qsmall} and \ref{qsmallpbad} prove the theorem for $n \ge 30$. When $q > \frac{\sqrt{n}}{2}$ and $\gcd(q,n) > 2$, Lemma \ref{qgcdbig} proves the theorem for $n \ge 1024$, and when $q > \frac{\sqrt{n}}{2}$ and $\gcd(q,n) \le 2$, Lemma \ref{rationalapprox} proves the theorem for $n > 10000$ and $24 j(n)^2 < \sqrt{n}$. So the only case in which checking triples up to $n = 10000$ will not complete the proof is the case $q > \frac{\sqrt{n}}{2}, \gcd(q,n) \le 2$.

\begin{rem}
By the bounds on the Jacobsthal function introduced in \S7,
$$j(n)^2 < 4^{1.4 \frac{\ln n}{\ln \ln n}} = n^{\frac{1.4 \ln 4}{\ln \ln n}}$$
(which we are comparing to $\frac{1}{24}\sqrt{n}$).
\end{rem}

We write $p_m\#$ for the product of the first $m$ primes and $\omega(m)$ for the number of distinct prime factors of $m$. Values of the function
$$H(n) := \max_{\omega(m) = n} j(m)$$
for $n \le 24$ have been computed by Hajdu and Saradha \cite{HS}. We will use their computed values, along with the obvious observation that $\omega(n) < m$ for $n < p_m\#$, to obtain bounds on $j(n)$ for small $n$ better than those presented in \S7. 
(From now on, every statement we make about the maximum value of $j(n)$ in some particular range can be attributed to their computations.)

\begin{rem}
\label{ninetwo}
For $n \ge p_{25}\# > 2.3 \times 10^{36}$, we have $\ln \ln n > 4.4$, and so
$$j(n)^2 < n^{\frac{1.4 \ln 4}{\ln \ln n}} < n^{2/4.4} < \frac{1}{24} n^{1/2}.$$
So our bound holds for $n \ge p_{25}\#$. And for $n < p_{25}\#$, we have $\omega(n) \le 24$ and therefore $j(n) \le 236$. Then for $n > 24^2 (236)^4$ this bound holds, and for $n \le 24^2 (236)^4$ we have $\omega(n) \le 11$, and therefore $j(n) \le 58$. So for $n > 24^2 (58)^4$ this bound holds, and for $n \le 24^2 (58)^4$ we have $\omega(n) \le 10$, so $j(n) \le 46$. Then this bound holds for $n > 24^2 (46)^4$, and for $n \le 24^2 (46)^4$ we have $\omega(n) \le 9$, so $j(n) \le 40$. Then this bound holds for $n > 24^2 (40)^4 = 1.47456 \times 10^9$, but the reduction process stops here, as we still have $\omega(n) \le 9$. 
\end{rem}

It therefore remains to check triples with $q > \frac{\sqrt{n}}{2}$ and $\gcd(q,n) \le 2$ up to $n = 1.47456 \times 10^9$. However, we will by no means check all of them individually.
Instead, we use the following strategy:
First of all, if we find some $k$ coprime to $n$ with $k < q$ so that $kq^{-1}r/zr$ is in $[0, \frac{n}{3}]$, $kq^{-1}/kz$ is a usable unit in $S_q \cap S_r$. As mentioned previously, any $n$ in this range has $\le 9$ prime factors, so if one can find 10 powers of distinct primes $c_1, \ldots, c_{10} < q$ so that $c_iq^{-1}r/c_izr \in [0, \frac{n}{3}]$, then at least one of the $c_i$ must be coprime to $n$, and so one of the $c_iq^{-1}/c_iz$ must be a usable unit in $S_q \cap S_r$. We will split the interval from 0 to $n$ into subintervals, the idea being to show that for almost all of the subintervals, if $q^{-1}r/zr$ is in this subinterval, then there are such $c_1, \ldots, c_{10}$. In order to get better bounds, we will divide into cases $r > \frac{4n}{5}$ and $\frac{2n}{3} < r \le \frac{4n}{5}$.

In the case $r > \frac{4n}{5}$, recall that by Propositions \ref{poweroftwo} and \ref{notfourdividesn} (the latter being applied with $m = 3$), if $\gcd(q,n) = 2$, then $S_q \cap S_r$ contains a usable unit. (As $q \ge \frac{\sqrt{n}}{2} > 50$ for $n > 10000$, the conditions $q > 2$ and $q > 8$ are satisfied.) So we may assume $\gcd(q,n) = 1$, and $S_q$ contains multiples of $q^{-1}$ up to $2q$. Now,
$[0, \frac{3n}{5}] \subset [0, 2r-n)$, so it suffices to use this larger interval. We will allow prime powers up to 80 (if $n \ge 6400$ and $q \ge \frac{\sqrt{n}}{2}$, then $2q \ge 80$). In this case, a computer check (partitioning $n$ into 10000 subintervals) shows that there are $c_1, \ldots, c_{10}$ unless $q^{-1}r$ is in $[\frac{4957n}{5000}, n)$. We note that this is the case considered in Proposition \ref{eightseven}, except that the upper bound on $m$ is replaced by $\frac{43n}{5000}$. However, the upper bound is only needed to ensure $j(n)m < \frac{3n}{5}$ (this being the $r > \frac{4n}{5}$ case), and as $j(n) \le 40$ given our bound on $n$, this is still satisfied. So we may apply the proposition to say that the \cite{MW} condition is satisfied for $q^{-1}r \in [\frac{4957n}{5000}, n)$ if $j(n) < \frac{3\sqrt{n}}{14}$. We will prove in Remark \ref{check} below that $j(n) < \frac{\sqrt{n}}{6}$ for $n \ge 10000$, so the case $q^{-1}r \in [\frac{4957n}{5000}, n)$ will not require any additional checking.

Now we consider the case $\frac{2n}{3} < r \le \frac{4n}{5}$. Here $q \ge \frac{n}{10} \ge 1000$ in the range being considered, but we must use the smaller target interval $[0, \frac{n}{3}]$. Allowing prime powers up to 1000 and partitioning $n$ into 12000 subintervals, the computer test reveals that there are $c_1, \ldots, c_{10}$ unless
$$q^{-1}r/zr \in \left[\frac{n}{3}, \frac{4005n}{12000}\right) \cup \left[\frac{5997n}{12000}, \frac{6007n}{12000}\right) \cup \left[\frac{2n}{3}, \frac{8005n}{12000}\right) \cup \left[\frac{11991n}{12000}, n\right)$$
We see that these exactly correspond to the cases $x = \frac{1}{3} + \epsilon, \frac{1}{2} \pm \epsilon, \frac{2}{3} + \epsilon, 1 - \epsilon$ considered in \S8. As in the case $r > \frac{4n}{5}$, the propositions of \S8 apply to our situation, as the only difference is that the upper bound on $m$ is replaced by the bounds given here, but as $j(n) \le 40$, the strongest upper bound on $m$ needed in those propositions, $3j(n)m < \frac{n}{3}$, is satisfied for $q^{-1}r/zr$ in the ranges shown above. So we may apply the propositions of \S8 to say that for $q^{-1}r/zr$ in these ranges, the \cite{MW} criterion is satisfied if $j(n) < \frac{\sqrt{n}}{6}, \frac{n}{216}$. As $\frac{\sqrt{n}}{6} \le \frac{n}{216}$ for $n \ge 1296$, i.e., in the range we are considering, it suffices to see when $j(n) < \frac{\sqrt{n}}{6}$.

\begin{rem}
\label{check}
For $n > 1.47456 \times 10^9$, by Remark \ref{ninetwo}, we have $j(n)^2 < \frac{\sqrt{n}}{24}$, so $j(n) < \frac{\sqrt{n}}{6}$ holds. Then for $n \le 2 \times 10^9$, $j(n) \le 40$, so for $n > 36(40)^2 = 57600$, this holds. For $n \le 57600$, we have $\omega(n) \le 6$, so $j(n) \le 22$. So for $n > 36(22)^2 = 17424$ this holds, and for $n \le 17424$ we have $\omega(n) \le 5$, and so $j(n) \le 14$. So $j(n) < \frac{\sqrt{n}}{6}$ holds for $n > 36(14)^2 = 7056$, and $n \le 7056$ is in the range that we plan to check anyway.
\end{rem}

So there is no need for additional checking in the $\frac{2n}{3} < r \le \frac{4n}{5}$ case either, i.e., it suffices to check only $n \le 10000$.

Checking all triples in this range besides those belonging to the six exceptional families listed in Theorem \ref{maintwo}, we find the following additional triples for which the \cite{MW} criterion is not satisfied: $(1,4,11)$, $(1,3,16)$, $(2,3,17)$, $(1,4,21)$, $(1,8,19)$, $(3,8,29)$, $(2,11,29)$. This is the list which appears in item 7 in Theorem \ref{maintwo}. For a summary of known results about these families and exceptional cases, the reader is referred to \S4.

\section{The family $p = 1, q = 4, r \equiv 7$ mod 8}
In this section, we prove that the triangles in family 3 of Theorem \ref{maintwo} do not have the lattice property (excluding Hooper's triangle, which has $r = 7 < 8 = \frac{2}{3}n$). As families 1--2 are known to be families of lattice triangles, families 4--6 have $r < \frac{3n}{4}$, and the exceptional triangles in item 7 have been excluded by the computer program of R\"uth, Delecroix, and Eskin \cite{RDE}, the consequence of this section will be Theorem \ref{realmain}, the classification of rational obtuse lattice triangles with obtuse angle $\ge \frac{3\pi}{4}$.

Our first observation is that the unfoldings of triangles in this family have a very simple form: the unfolding is in the shape of an $n$-pointed star (cf. \cite[Figure 1]{Ward}), which, by chopping off the points and reassembling them, can be thought of as an $n$-gon with four $\frac{n}{4}$-gons attached to its edges (cf. \cite[Figure 1]{Hooper}). (Each $\frac{n}{4}$-gon is attached to every fourth edge of the $n$-gon.) We easily identify the two horizontal cylinders shown in Figures 1 and 2. (In the images, only the relevant parts of the $n$-gon and attaching $\frac{n}{4}$-gons are shown. The black dots and dashes are intended to make edge identifications clear. The blue and red dashes indicate the center lines of the corresponding cylinders.)

\begin{figure}
\centering
\includegraphics[scale=.4]{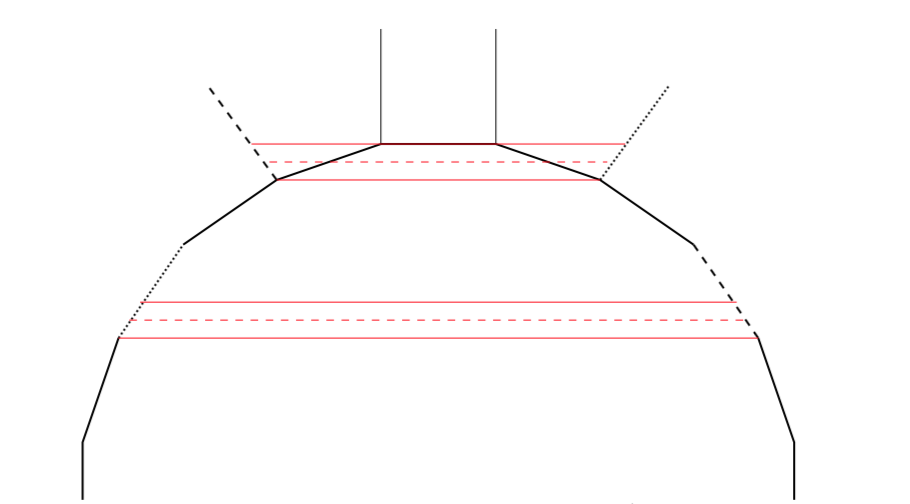}
\caption{``Top" cylinder}
\end{figure}

\begin{figure}
\centering
\includegraphics[scale=.45]{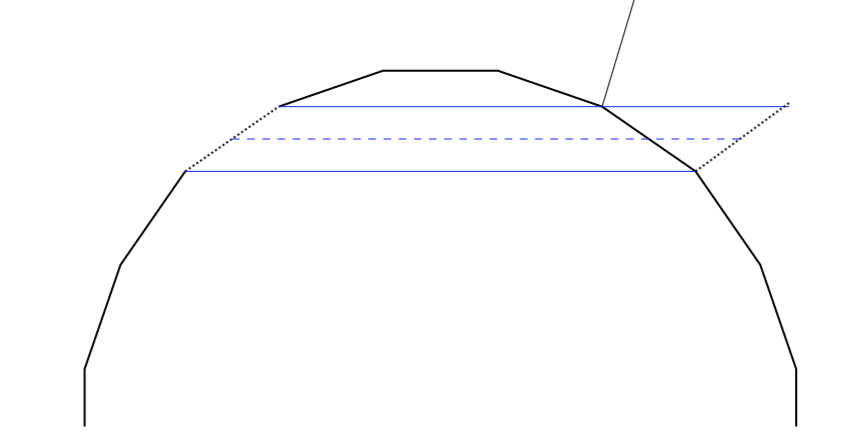}
\caption{``Bottom" cylinder}
\end{figure}

Letting $\alpha := \frac{(n-2)\pi}{n}$ be the interior angle of a regular $n$-gon of side length 1, one can calculate that the ``top" cylinder has height and circumference
$$h_t = \sin(\alpha), \,\,\, c_t = 2 - 4\cos(\alpha) + 2\cos(2\alpha) - 2 \cos(3\alpha)$$
and the ``bottom" cylinder has height and circumference
$$h_b = -\sin(2\alpha), \,\,\, c_b = 1 - 2\cos(\alpha) + 2\cos(2\alpha)$$
Comparing the moduli,
$$\frac{h_b}{c_b}/\frac{h_t}{c_t} = 4\cos^2(\alpha)$$
We note that $4\cos^2(\alpha) \in \Q$ would imply $\cos(2\alpha) \in \Q$, and as $2\alpha$ is a rational angle, it is a well-known fact that this would imply $2\alpha$ is a multiple of $\frac{\pi}{3}$ or $\frac{\pi}{2}$, i.e., $\alpha$ must certainly be a multiple of $\frac{\pi}{12}$. For the first triangle in this family, Hooper's triangle, we have $n = 12$, and $4\cos^2(\alpha) = 3$, as computed in \cite[Equation 8]{Hooper}. However, for all larger $n = 12 + 8x$ ($x \ge 1$), $\alpha = \frac{(n-2)\pi}{n}$ is not a multiple of $\frac{\pi}{12}$, and so the ratio of the moduli is irrational. By \cite[Remark on p. 582]{Veech}, this implies that the triangles of this family with $n > 12$ do not have the lattice property, and as explained at the beginning of this section, this argument completes the classification of rational lattice triangles with obtuse angle $\ge \frac{3\pi}{4}$.

\end{document}